\documentclass{amsart}
\usepackage{}
\usepackage{color}
\usepackage{amssymb}
\usepackage{amsthm}

\newcommand{\tF}{\tilde{F}}

\newcommand{\ol}{\overline}
\newcommand{\ul}{\underline}

\newtheorem{theorem}{Theorem}[section]
\newtheorem{lemma}[theorem]{Lemma}

 \theoremstyle{definition}
\newtheorem{definition}[theorem]{Definition}

\theoremstyle{remark}
\newtheorem{remark}[theorem]{Remark}

\numberwithin{equation}{section}

\begin{document}

\title[Dirichlet Problem]
{The Dirichlet problem for prescribed curvature equations of $p$-convex hypersurfaces}

\author{Weisong Dong}
\address{School of Mathematics, Tianjin University, 135 Yaguan Road,
Jinnan, Tianjin, China, 300354}
\email{dr.dong@tju.edu.cn}

%\date{}

\begin{abstract}

In this paper, we study the Dirichlet problem for $p$-convex hypersurfaces with prescribed curvature.
We prove that there exists a graphic hypersurface satisfying the prescribed curvature equation with homogeneous boundary condition.
An interior curvature estimate is also obtained.
%Moreover, the right hand side function of the equation is only assumed to be positive which may depend on the unit normal vector of the %hypersurface.

%\emph{Mathematical Subject Classification (2020):} 53C42; 53A07.


\end{abstract}

\maketitle

\section{Introduction}

In this paper, we are interested in finding a graphic hypersurface of prescribed curvature on a bounded domain $\Omega \subset \mathbb{R}^{n}$.
Suppose that the hypersurface
$M_u := \{(x, u(x) ): x\in \Omega\} \subset \mathbb{R}^{n+1}$ is given by the graph of
a smooth function $u:\Omega \rightarrow \mathbb{R}$. Denote by $\kappa[M_u] = (\kappa_1, \cdots, \kappa_n)$ the principal curvatures of $M_u$
with respect to the downward unit normal of $M_u$.
Given an integer $p$, where $1\leqslant p\leqslant n$, it is interesting in geometry analysis to find
a graphic hypersurface $M_u$ with its principal curvatures satisfying the following equation
\begin{equation}
\label{eqn}
\Pi_{1\leqslant i_1 < \cdots < i_p\leqslant n} (\kappa_{i_1} + \cdots + \kappa_{i_p}) = f (x, u, \nu)\;  \mbox{in}\; \Omega,
\end{equation}
and with the prescribed homogeneous boundary condition as below
\begin{equation}
\label{eqn-b}
u= 0\; \mbox{on}\; \partial \Omega,
\end{equation}
where $\nu$ is the upward unit normal vector to the graphic hypersurface at $X := (x, u(x))$ and $f(x,z,\nu) > 0$ is a smooth function defined on $\Omega\times \mathbb{R} \times \mathbb{R}^n$.
For convenience, we denote that
$\mathcal{M}_u^p := \Pi_{1\leqslant i_1 < \cdots < i_p\leqslant n} (\kappa_{i_1} + \cdots + \kappa_{i_p})$.
We also simply write $\mathcal{M}_u$ when there is no ambiguity.

The equation \eqref{eqn} is a fully nonlinear partial differential equation for $p < n$.
It is natural to consider the problem in the class of $p$-convex hypersurfaces
so that the equation is elliptic.
Recall that
a $C^2$ regular hypersurface $M_u$ is called (strictly) $p$-convex if
$\kappa[M_u] (x)$ satisfies, at each point $x \in \Omega$, that
\[
\kappa_{i_1} + \cdots + \kappa_{i_p} \geqslant \;(>)\; 0, \: \forall \;  1\leqslant i_1 < \cdots < i_p\leqslant n.
\]
Accordingly, we say that the $C^2$ function $u$ is \emph{admissible}.
The notion of $p$-convexity can be dated back to Wu \cite{Wu}.
Sha in \cite{Sha1,Sha2} studied Riemannian manifolds with $p$-convex boundaries and proved a complete characterization for such
Riemannian manifolds.
Recently, there were some new discoveries in $p$-convex geometry made by Harvey and Lawson.
We refer the reader to \cite{HL09,HL12} for these results and also to a nice survey \cite{HL13}.

There is a vast literature on the existence of closed Weingarten hypersurfaces of codimension one in the Euclidean space.
The Gausssian curvature case corresponding to $p=1$ in \eqref{eqn} was studied by Oliker \cite{Oliker}.
The mean curvature equation corresponding to $p=n$ in \eqref{eqn} was studied by Bakelman-Kantor \cite{BK} and Treibergs-Wei \cite{TreW}.
Caffarelli-Nirenberg-Spruck in \cite{CNS4} was concerned with prescribed curvature equations
of very general type, and it was then generalized to Riemannian manifolds by Gerhardt in \cite{Gerhardt}.
The existence of a starshaped $(n-1)$-convex hypersurface of prescribed curvature \eqref{eqn} was proved by Chu-Jiao \cite{CJ}.
The case $p\geqslant \frac{n}{2}$ was explored in a subsequent work by the author in \cite{Dong}.
In complex settings, when $p=n-1$, the operator appears in
the Gauduchon conjecture which was
solved by Sz\'ekelyhidi-Tosatti-Weinkove \cite{S-T-W}.
Some previous work on this topic can be found in Tosatti-Weinkove \cite{TW1,TW2} and Fu-Wang-Wu \cite{FWW1,FWW2}.
A counterpart to \eqref{eqn} that was studied extensively is the higher order mean curvature equation
\begin{equation}
\label{sigmak}
\sigma_k (\kappa) = f(X, \nu(X)),\; \forall X \in M,
\end{equation}
where $1\leqslant k \leqslant n$ and $\sigma_k$ is the $k$-th elementary symmetric function.
In warped product manifolds, Chen-Li-Wang \cite{CLW} established the curvature estimate
for convex hypersurfaces satisfying \eqref{sigmak}.
%The convexity allows that the right hand function in their equation can depend on $\nu$.
%It is worth to mention that here we do not impose any convexity conditions on $f$ to derive the curvature estimate.
Jin-Li \cite{JL} and Li-Oliker \cite{LO} investigated the starshaped hypersurfaces in hyperbolic and elliptic space
respectively.
For the study in Minkowski space, see Li \cite{LiAnmin}, Wang-Xiao \cite{WX} and Ren-Wang-Xiao \cite{RWX}.
We also refer the reader to Guan-Zhang \cite{GZ} for a class of curvature equations arising
from convex geometry, and to Guan-Lin-Ma \cite{GLM} and Guan-Li-Li \cite{GLL} for hypersurfaces with prescribed curvature measures.

The Dirichlet problem for Weingarten hypersurfaces has also attracted attention of many authors, 
due to its role in solving a Plateau type problem for locally convex Weingarten hypersurfaces.
For its own sake, there still are many works, such as Caffarelli-Nirenberg-Spruck \cite{CNS5}, Lions \cite{Lions}, Ivochkina \cite{I1,I2}, Ivochkina-Lin-Trudinger \cite{ILT}, Lin-Trudinger \cite{LT} and Trudinger \cite{Trudinger-arma}.
We remark that in \cite{ILT} the Dirichlet problem with non-homogeneous boundary data was considered for smooth solutions
and in \cite{Trudinger-arma} it was investigated in the context of viscosity solutions.
In this paper, we are focused on the Dirichlet problem of homogeneous boundary data.
We shall return to the non-homogeneous boundary case in future work.
Now,
we recall the $p$-convex cones introduced by Harvey and Lawson \cite{HL13}.
\begin{definition}
\label{def}
Let $p\in \{1, \cdots, n\}$. The $\mathcal{P}_p$ cone is a subset in $\mathbb{R}^n$ with element
$(\lambda_1, \cdots, \lambda_n)$ such that $\lambda_{i_1} + \cdots + \lambda_{i_p} > 0$
for all $1\leqslant i_1 < \cdots < i_p\leqslant n$.
The $P_p$ cone of symmetric $n \times n$ matrices is defined as:
$A \in P_p$ if the $n$-tuple of the eigenvalues of matrix $A$ is in $\mathcal{P}_p$.
We call $A$ is $p$-positive if $A\in P_p$.
\end{definition}
Denote by $(\kappa_1^b, \cdots, \kappa_{n-1}^b)$ the principal curvatures of the boundary $\partial \Omega$ taken
with respect to the exterior unit normal to $\partial \Omega$ (see Section 14.6 in \cite{GT}).
We call $\Omega$ is \emph{strictly convex} if that
$(\kappa_1^b(x), \cdots, \kappa_{n-1}^b(x)) \in \mathcal{P}_{1} \subset \mathbb{R}^{n-1}$ for every $x\in \partial \Omega$.
Now, we state our main result.
\begin{theorem}
\label{thm1}
Let $\Omega$ be a strictly convex bounded domain in $\mathbb{R}^n$ with smooth boundary $\partial \Omega$
and $f = f (x, z, \nu) \in C^{\infty} (\bar \Omega \times \mathbb{R} \times \mathbb{S}^n)$ be a positive function with $f_z \geqslant 0$.
Suppose that $p\geqslant \frac{n}{2}$.
Assume that there is an admissible subsolution $\ul{u} \in C^2(\bar \Omega)$ satisfying
\begin{equation}
\label{subsolution}
\mathcal{M}_{\ul u}^p \geqslant f (x, \ul u, \ul \nu) \;\mbox{in}\; \Omega \;\; \mbox{and}\;\;  \ul u = 0 \; \mbox{on}\; \partial \Omega,
\end{equation}
where $\ul \nu$ denotes the upward unit normal vector to the graphic hypersurface $M_{\ul u}$ at $\ul X := (x, \ul u(x))$.
Then there exists a unique admissible solution $u\in C^\infty (\bar \Omega)$ to \eqref{eqn} and \eqref{eqn-b}.
\end{theorem}

\begin{remark}
The endpoint case $p=n$ is excluded from our discussion, since this case belongs to the realm of quasilinear elliptic equations.
The case $p=1$ is also well known.
Jiao-Sun \cite{JS} proved the theorem for $p=n-1$ recently.
%The assumption $p\geqslant \frac{n}{2}$ is only used in Lemma \ref{key-1}.
It remains an interesting question to prove the theorem for $1< p < \frac{n}{2}$.
\end{remark}

The approach to solve the fully nonlinear elliptic equation \eqref{eqn} is the well-known
continuity method and degree theory. So the main task of this paper is to establish \emph{a priori} estimates for
admissible solutions up to second order derivatives as higher order estimates follow from the Evans-Krylov theorem and Schauder theory.
By the assumption of the existence of a subsolution, $C^0$ and $C^1$ estimates are easy to derive.
The most difficult part is to establish \emph{a priori} $C^2$ estimates, including the global $C^2$ estimate
and the boundary $C^2$ estimate.
Moreover, since $f$ depends on $\nu$, the bad term $-Cv_{11}$ appears in the inequality \eqref{S2-3}
when applying the maximum principle to the test function $G$ in Section 3,
which arise from differentiating the equation \eqref{eqn} twice (see \eqref{diff-eqn-2}).
This also cause trouble for us to deal with third order terms.
To overcome these difficulties, one can impose various assumptions on $f$ as in \cite{I1,I2,GG,GJ}.
%For our equation, there are no extra assumptions on $f$ except that it is positive.
%The main difficulty in establishing \emph{a priori} estimates is to prove the curvature estimate and the boundary $C^2$ estimate.
%To establish the curvature estimate, the crucial step is to overcome the bad third order terms.
Due to the special nature of the operator of equation \eqref{eqn},
we use idea from \cite{CJ,Dong} to control $-Cv_{11}$ by good terms
without imposing any extra conditions on $f$.
Then the bad third order terms can be eliminated by the good term $-F^{1i,i1}$.
See Lemma \ref{key-1} and Lemma \ref{lem4}.
As to the boundary $C^2$ estimate, the issue is to estimate the tangential-normal derivatives of solutions.
Fortunately, we can use idea from Ivochkina \cite{I1} to tackle this.
Actually, we have enough good third order terms to derive an interior curvature estimate.
\begin{theorem}
Let $\Omega$ be a bounded domain in $\mathbb{R}^n$.
Suppose that $u\in C^4 (\Omega) \cap C^2 (\bar \Omega)$ is an admissible solution to \eqref{eqn} and \eqref{eqn-b}.
Then there is a positive constant $C$ and $\beta$ such that the second fundamental form $h$ of $M_u$ satisfies
\begin{equation}
\sup_{ \Omega} (-u)^\beta|h| \leqslant C,
\end{equation}
where $C$ and $\beta$ depend on $n$, $p$, $|u|_{C^1}$, $\inf \tilde f$ and $|f|_{C^2}$.
\end{theorem}

Some applications of the interior curvature estimate can be found in Sheng-Urbas-Wang \cite{SUW}.
The above theorem holds for $p\geqslant \frac{n}{2}$ if $f \equiv f (x, u ,\nu)$ and
for $1\leqslant p \leqslant n$ if $f \equiv f(x,u)$
as explained below.
The assumption $p\geqslant \frac{n}{2}$ in Theorem \ref{thm1} is only used to prove Lemma \ref{key-1}
to control $-Cv_{11}$ in the inequality \eqref{S2-4} (and in \eqref{In-S2-2} for the interior curvature estimate).
If the right hand side function $f$ in the equation \eqref{eqn} does not depend on $\nu$,
the term $-Cv_{11}$ does not appear in \eqref{S2-4} (also not appear in \eqref{In-S2-2}),
for which we refer the reader to examine the inequalities \eqref{diff-eqn-2} and \eqref{S2-3}. Therefore, we have the following result.

\begin{theorem}
\label{thm2}
Let $\Omega$ be a strictly convex bounded domain in $\mathbb{R}^n$ with smooth boundary $\partial \Omega$
and $f = f (x, z) \in C^{\infty} (\bar \Omega \times \mathbb{R})$ be a positive function with $f_z \geqslant 0$.
Suppose $1\leqslant p \leqslant n$.
Assume that there is an admissible subsolution $\ul{u} \in C^2(\bar \Omega)$ satisfying
\begin{equation}
\label{subsolution-2}
\mathcal{M}_{\ul u}^p \geqslant f (x, \ul u) \;\mbox{in}\; \Omega \;\; \mbox{and}\;\;  \ul u = 0 \; \mbox{on}\; \partial \Omega,
\end{equation}
where $\ul \nu$ denotes the upward unit normal vector to the graphic hypersurface $M_{\ul u}$ at $\ul X := (x, \ul u(x))$.
Then there exists a unique admissible solution $u\in C^\infty (\bar \Omega)$ to \eqref{eqn} and \eqref{eqn-b}.
\end{theorem}

Before we end the introduction, let's review a conjecture on curvature estimates
for hypersurfaces with prescribed higher order mean curvature as it is closely related to our topic.
The Gaussian curvature and the mean curvature correspond to $k=n$ and $k=1$ in \eqref{sigmak} respectively.
Ren-Wang conjectured that the curvature estimate holds for $k> \frac{n}{2}$ in \cite{RW2,RW3}.
While the case $k=n$ proved a long time ago by Caffarelli-Nirenberg-Spruck \cite{CNS1},
it was only recently that Ren-Wang  \cite{RW1,RW2} proved the conjecture for $k\geqslant n-2$.
Guan-Ren-Wang in \cite{GRW} established the curvature estimate when $k=2$,
and in the same paper they also proved
the curvature estimate for convex solutions of equation \eqref{sigmak}
when $3 \leqslant k \leqslant n$.

The paper is organized as follows. In Section 2, we recall some properties of the operator from \cite{Dinew}
and prove some key lemmas which are crucial to the global and boundary $C^2$ estimate.
Then, we prove the global and boundary $C^2$ estimate in Section 3 and Section 4 respectively.
In Section 5, we provide a proof of the gradient estimate.

\textbf{Acknowledgements}:
%The author would like to thank the hospitality of YMSC, Tsinghua University, where the final version of the paper was done.
%We would like to thank the anonymous referees for
%helpful comments
%and
%especially for pointing out a mistake.
The author is supported by the National Natural Science Foundation of China, No.
11801405.

\section{Preliminaries}

Let $M$ be a hypersurface in $\mathbb{R}^{n+1}$ given by the graph of $u : \Omega \rightarrow \mathbb{R}$.
%At a point $X = (x, u(x))$, $\{e_1, \cdots, e_n\}$ is a local frame on $M$, where $e_i = (0,\cdots, 1, \cdots, u_i)$.
Then the induced metric $g$ and second fundamental form $h$ of $M$ are given by
\[
g_{ij} = \delta_{ij} + u_iu_j \;\mbox{and}\; h_{ij} = \frac{u_{ij}}{\sqrt{1+ |Du|^2}}.
\]
The upward unit normal vector is
\[
\nu = \frac{(-Du,1)}{\sqrt{1+ |Du|^2}}.
\]
The principal curvatures $\kappa = (\kappa_1, \cdots, \kappa_n)$ of the hypersurface $M$ are the eigenvalues of $h_{ij}$
with respect to $g_{ij}$, i.e. the eigenvalues of the following matrix
\begin{equation}
\label{matrix'}
\frac{1}{w} \Big(I - \frac{Du\otimes Du}{w^2}\Big) D^2 u,
\end{equation}
where $w = \sqrt{1+ |Du|^2}$ and $g^{ij} = \delta_{ij} - \frac{u_iu_j}{w^2}$ is the inverse of $g_{ij}$.
Note that the above matrix generally is not symmetric.
As in \cite{CNS5}, we consider the following matrix
\begin{equation}
\label{matrix}
a_{ij} = \frac{1}{w} \gamma^{ik} u_{kl} \gamma^{lj},
\end{equation}
where $\gamma^{ik} = \delta_{ik} - \frac{u_i u_k}{w(1+w)}$ is the inverse matrix of $\gamma_{ij} = \delta_{ij} + \frac{u_i u_j}{1+w}$
and $\gamma_{ij}$ is the square root of $g_{ij}$.

%Now we recall a few notations from \cite{JS}.
%For a $n\times n$ symmetric matrix $r$ and $p \in \mathbb{R}^n$, denote by $\lambda(p,r)$ the eigenvalues of the matrix below
%\[
%\Big(I - \frac{p \otimes p}{1 + |p|^2} \Big) r.
%\]

For convenience, we introduce the following notations
\[
F (a_{ij}) := F (\kappa) = \Pi_{1\leqslant i_1 < \cdots < i_p\leqslant n} (\kappa_{i_1} + \cdots + \kappa_{i_p}) \; \mbox{and} \; \tilde F = F^{\frac{1}{C_n^p}},
\]
where $C_n^p = \frac{n!}{p! (n-p)!}$.
Equation \eqref{eqn} then can be written as
\begin{equation}
\label{eqn'}
\tilde F (a_{ij}) := \tilde F(\kappa) = \tilde f(x,u, \nu),
\end{equation}
where $\kappa = (\kappa_1, \cdots, \kappa_n)$ and $\tilde f = f^{ \frac{1}{C_n^p} }$.
It is also convenient to rewrite the equation as below
\begin{equation}
\label{G-eqn}
G(D^2 u, Du) := \tilde F( a_{ij} ) = \tilde f,
\end{equation}
where $a_{ij}$ is defined in \eqref{matrix}.
Denote
\[
F^{ij} = \frac{\partial F}{\partial a_{ij}}, \;
F^{ij,kl} = \frac{\partial^2 F}{\partial a_{ij} \partial a_{kl} },\;
G^{ij} = \frac{\partial G}{\partial u_{ij}},\;
G^{i} = \frac{\partial G}{\partial u_i}.
\]
We remark that $G$ is concave and homogeneous one with respect to $D^2u$ since
$\tilde F$ is concave and homogeneous one with respect to $a_{ij}$. See \cite{Dinew}.
The equation is elliptic as the matrix $\{ \frac{\partial \tF}{\partial a_{ij}} \}$ is positive definite for $\{a_{ij}\} \in P_p$.
It is easy to verify that
\[
\frac{1}{w} \sum \tF^{ii} \geqslant \sum G^{ii} \geqslant \frac{1}{w^3} \sum \tF^{ii}.
\]
By a straightforward calculation, one get the following lemma. See formula (2.21) in Lemma 2.3 of \cite{GS}.
\begin{lemma}[\cite{GS}]
\label{GS}
We have
\begin{equation}
\label{GS-equality}
G^s = - \frac{u_s}{w^2} \sum_i \frac{\partial \tF}{\partial \kappa_i} \kappa_i
- \frac{2}{w(1+w)} \sum_{ij} \frac{\partial \tF }{\partial a_{ij}} a_{it} (wu_t \gamma^{sj} + u_j\gamma^{ts}).
\end{equation}
\end{lemma}
%Direct calculations show that
%$ \tilde F^{ij} = \frac{1}{C_n^p} F^{\frac{1}{C_n^p} - 1} F^{ij}$
% and
% \[
% \tilde F^{ij,kl} = \frac{1}{C_n^p} F^{\frac{1}{C_n^p} - 1} F^{ij,kl}
% + \frac{1}{C_n^p} \Big( \frac{1}{C_n^p} - 1 \Big) F^{\frac{1}{C_n^p} - 2} F^{ij} F^{kl}.
% \]

Next, we calculate the derivatives of $F$.
The calculations are carried out at a point $X_0$ on the hypersurface $M_u$
where $a_{ij}$ is diagonal.
Note that $F^{ij} $ is also diagonal at $X_0$ and we have the following formulas
\[
F^{kk} = \frac{\partial F}{\partial \kappa_k} =
\sum_{k\in\{i_1, \cdots, i_p\} } \frac{F (\kappa)}{\kappa_{i_1} + \cdots + \kappa_{i_p} },
\]
for which we refer to Lemma 1.10 in \cite{Dinew}.
We have formulas for the second order derivatives of $F$ at $X_0$ as below
\[
F^{kk,ll}= \frac{\partial^2 F}{\partial \kappa_k \partial \kappa_l} =
\sum_{ \substack{ k\in\{i_1, \cdots, i_p\} \\ l\in \{j_1, \cdots, j_p\} \\ \{i_1, \cdots, i_p\} \neq \{j_1, \cdots, j_p\}} }
\frac{F (\kappa)}{ (\kappa_{i_1} + \cdots + \kappa_{i_p}) (\kappa_{j_1} + \cdots + \kappa_{j_p}) },
\]
and, for $k\neq r$,
\[
F^{kr, r k} = \frac{F^{k k} - F^{r r} }{\kappa_k - \kappa_r}
=
%- \sum_{k, r \notin\{i_1, \cdots, i_{p-1}\}}
% \frac{F(\kappa)}{ (\kappa_k + \kappa_{i_1} + \cdots + \kappa_{i_{p-1}} ) ( \kappa_r + \kappa_{i_1} + \cdots + \kappa_{i_{p-1}} )}.
- \sum_{ \substack{ k\notin\{i_1, \cdots, i_p\} \ni r  \\ r\notin \{j_1, \cdots, j_p\} \ni k
\\ \{i_1, \cdots, i_p\} \backslash \{r\} = \{j_1, \cdots, j_p\} \backslash \{k\}
}  }
\frac{ F(\kappa) }{( \kappa_{i_1} + \cdots + \kappa_{i_p} )(\kappa_{j_1} + \cdots + \kappa_{j_p})}.
\]
Otherwise, we have $F^{ij,kl} = 0$.
See Lemma 1.12 in \cite{Dinew} for the above formulas, or see \cite{Dong}. These formulas can also be easily obtained from Theorem 5.5 in \cite{B}.
The following properties of the function $F$ which are very similar to the properties of $\sigma_k$ were proved by Dinew \cite{Dinew}.

\begin{lemma}[\cite{Dinew}]
\label{Dinew1}
Suppose that the diagonal matrix $A = diag(\lambda_1, \cdots, \lambda_n)$ belongs to $P_p$ and
that $\lambda_1 \geqslant \cdots \geqslant \lambda_n$.
Then,
\begin{itemize}
\item [(1)] $\tF^{11} (A) \lambda_1 \geqslant \frac{1}{n} \tF (A);$
\item [(2)] $\sum_{k=1}^n \tF^{kk} (A) \geqslant p;$
\item [(3)] $\sum_{k=1}^n F^{kk} (A) \lambda_{k} = C_n^p F(A);$
\item [(4)] there is a constant $\theta = \theta(n, p)$ such that $F^{jj} (A) \geqslant \theta \sum F^{ii}$ for all  $j \geqslant n-p+1.$
\end{itemize}
\end{lemma}

The proof of the above Lemma can also be found in the Appendix of \cite{Dong}.
Next, we prove a few lemmas which were used to derive second order estimates.

%\begin{lemma}[\cite{Dinew}]
%\label{Dinew2}
%Suppose $A \in P_p$ be an arrowhead matrix. If $a_{11} \leqslant -c < 0$ for some $c> \frac{2C_{n-1}^p}{n} (F (A))^{\frac{1}{n}}$, then
%we have
%\[
%F^{11} (A) \geqslant \theta \sum_{i} F^{ii} (A)
%\]
%for some $\theta >0$ depending on $n$, $p$ and $c$.
%\end{lemma}

\begin{lemma}
\label{growth}
For every $C>0$ and every compact set $K\subset\mathcal{P}_p$, there exists a number $R = R(C,K)$ such that
the inequality
\begin{equation}
\label{F5}
F(\lambda_1, \cdots, \lambda_n +R) \geqslant C
\end{equation}
holds for all $\lambda=(\lambda_1, \cdots, \lambda_n) \in K$.
\end{lemma}

\begin{proof}
%Denote by $\mathcal{P}_{p}'$ and $K'$ the projection of the cone $\mathcal{P}_p$ and the compact set $K$ to $\mathbb{R}^{n-1}$
%respectively.
%Note that $\mathcal{P}_{p}'$
%is exactly the cone $\mathcal{P}_p \subset \mathbb{R}^{n-1}$ and $K' \subset \mathcal{P}_p \subset \mathbb{R}^{n-1}$ is compact.
Since $K$ is compact, there is a positive constant $\varrho = C(K)$ such that,
for all $\lambda \in K $ and $\{i_1, \cdots, i_p\} \subset \{1, \cdots, n-1\}$,
\[
\lambda_{i_1} + \cdots + \lambda_{i_{p}} \geqslant \varrho.
\]
Therefore, for all $\lambda \in K$, we have
\begin{equation}
\begin{aligned}
&\; F(\lambda_1, \cdots, \lambda_n +R) \\
= &\; \Pi_{1\leqslant i_1 < \cdots < i_{p-1} \leqslant n-1} (\lambda_{i_1} + \cdots + \lambda_{i_{p-1}} + \lambda_n + R)\\
&\; \times \Pi_{1\leqslant i_1 < \cdots < i_{p} \leqslant n-1} (\lambda_{i_1} + \cdots + \lambda_{i_{p}})\\
\geqslant &\; R^{C_{n-1}^{p-1}} \times \Pi_{1\leqslant i_1 < \cdots < i_{p} \leqslant n-1} (\lambda_{i_1} + \cdots + \lambda_{i_{p}})\\
\geqslant &\; R^{C_{n-1}^{p-1}} \varrho^{C_{n-1}^p},
\end{aligned}
\end{equation}
which goes to infinity when $R$ goes to infinity.
Thus, the lemma is proved.
\end{proof}

The following two lemmas were essentially proved in \cite{Dong} (see Lemma 3.2 and Lemma 3.4),
which were used to deal with the bad third order terms.
Here, we restate them in the following form.
\begin{lemma}
\label{key-1}
Suppose $p\geqslant \frac{n}{2}$ and $A = diag (\lambda_1, \cdots, \lambda_n) \in P_p$
with $\lambda_1 \geqslant \cdots \geqslant \lambda_n$.
If $\lambda_n \geqslant - \delta \lambda_1$ where $0 < \delta \leqslant \frac{1}{2(p-1)}$,
we have
\[
 C \sum F^{ii}(A) \geqslant \lambda_1
\]
for $C$ sufficiently large.
\end{lemma}

\begin{proof}
We divide the proof into two cases.

Case 1.
Suppose $\lambda_{n-p+1} + \lambda_{n-p+2} + \cdots + \lambda_n < \frac{1}{\lambda_1}$.
Since
\[
F^{nn} \geqslant \frac{F(\kappa)}{\lambda_{n-p+1} + \lambda_{n-p+2} + \cdots + \lambda_n},
\]
we see that
\[
F^{nn} \geqslant c_0 \lambda_1
\]
for some $c_0 > 0$ depending on $\inf f$. The lemma follows for sufficiently large $C$ depending on $\inf f$.

Case 2. Suppose $\lambda_{n-p+1} + \lambda_{n-p+2} + \cdots + \lambda_n \geqslant \frac{1}{\lambda_1}$.
For a fixed $(p-1)$-tuple $2\leqslant i_1 < \cdots < i_{p-1} \leqslant n$, we have
\[
\lambda_1 + \lambda_{i_1} + \cdots + \lambda_{i_{p-1}} \geqslant \frac{\lambda_1}{2}
\]
by our assumption.
Hence, we have
\[\begin{aligned}
F^{nn} \geqslant &\; \mathop{\Pi}\limits_{2\leqslant i_1 < \cdots < i_{p-1} \leqslant n } ( \lambda_1 + \lambda_{i_1} + \cdots + \lambda_{i_{p-1}} ) \\
& \times \mathop{\Pi}\limits_{ \substack{ 2\leqslant i_1 < \cdots < i_{p} \leqslant n \\ (i_1, \cdots, i_p) \neq (n-p+1, \cdots, n) } }
( \lambda_{i_1} + \lambda_{i_2} + \cdots + \lambda_{i_{p}} )\\
\geqslant &\; [ \frac{ \lambda_1}{2} ]^{C_{n-1}^{p-1}} [\frac{1}{\lambda_1}]^{C_{n-1}^p - 1}.
\end{aligned}\]
For $p\geqslant \frac{n}{2}$, a direct calculation shows that
\[
C_{n-1}^{p-1} - C_{n-1}^{p} = \frac{(n-1)\cdots(n-p+1)}{(p-1)!} \Big( 1 - \frac{n-p}{p} \Big) \geqslant 0.
\]
Therefore, we obtain
\[
2^{C_{n-1}^{p-1} } F^{nn} \geqslant  \lambda_1 .
\]
\end{proof}

\begin{lemma}
\label{lem4}
Suppose $A = diag (\lambda_1, \cdots, \lambda_n) \in P_p$
with $\lambda_1 \geqslant \cdots \geqslant \lambda_n$.
For any given $1> \epsilon > 0$, there is a $\delta = \delta(\epsilon) > 0$ such that,
if $\lambda_n \geqslant - \delta \lambda_1$,
the inequality
\[
-2 F^{1i,i1}(A) + (1 + \epsilon) \frac{F^{11}(A)}{\lambda_1} \geqslant (1 + \epsilon) \frac{F^{ii}(A)}{\lambda_1}
\]
holds for $i= 2,3, \cdots n$.
\end{lemma}

\begin{proof}
Observe that
$\frac{F^{ii}}{\lambda_1} = \frac{\lambda_1 -\lambda_i}{\lambda_1}(- F^{1i,i1}) + \frac{F^{11}}{\lambda_1}$, from which we obtain that
\[
\frac{F^{ii}}{\lambda_1} \leqslant - F^{1i,i1} + \frac{F^{11}}{\lambda_1} \; \mbox{for} \; \lambda_i \geqslant 0.
\]
Since $-\lambda_n \leqslant \delta \lambda_1$, we have
\[
\frac{F^{ii}}{\lambda_1} \leqslant - (1+\delta) F^{1i,i1} + \frac{F^{11}}{\lambda_1} \; \mbox{for}\; \lambda_i \leqslant 0.
\]
By the above two inequalities, we get the desired inequality for sufficiently small $\delta$.
\end{proof}

\section{Global and interior $C^2$ estimates}

Let $\langle \cdot, \cdot\rangle$ be the inner product in $\mathbb{R}^{n+1}$
and $\{e_1, \cdots, e_n, e_{n+1}\}$ be the standard basis of $\mathbb{R}^{n+1}$.
Then $X = (x, u(x)) = x_1 e_1 + \cdots + x_n e_n + u(x) e_{n+1}$
and $|X|^2 = \langle X, X\rangle = \sum_i x_i^2 + u(x)^2$.

\subsection{Global $C^2$ estimates}

Now we prove the global $C^2$ estimate. The main idea which is from \cite{Dong} is to control the bad third order terms.

\begin{theorem}
Let $\Omega$ be a bounded domain in $\mathbb{R}^n$.
Suppose that $u\in C^4 (\Omega) \cap C^2 (\bar \Omega)$ is an admissible solution to \eqref{eqn}.
Then there is a positive constant $C$ such that
\begin{equation}
\label{C2}
\sup_{ \Omega} |D^2 u| \leqslant C (1 + \sup_{\partial \Omega} |D^2 u|),
\end{equation}
where $C$ depends on $n$, $p$, $|u|_{C^1}$, $\inf f$ and $|f|_{C^2}$.
\end{theorem}

\begin{proof}
By the relation between $h_{ij}$ and $u_{ij}$, to prove the global $C^2$ estimate it is equivalent to prove the curvature estimate.
Let $a$ be a positive constant depending on $|Du|_{C^0}$ such that $\frac{1}{w} \geqslant 2a$.
Consider the following test function
\[
G (x, \xi) =  (\frac{1}{w} - a)^{-1 } e^{\frac{A}{2} |X|^2} h_{\xi\xi},
\]
where $X \in M_u$, $\xi \in T_X M_u$ is a unit vector and $A$ is a large constant to be determined later.
Suppose the maximum of $G$ is achieved at a point $X_0 = (x_0, u(x_0))\in M_u$ and $\xi_0 \in T_{X_0} M_u$.
We choose new orthonormal vectors $\{\epsilon_1, \cdots, \epsilon_n, \epsilon_{n+1}\}$ at $X_0$ such that
$\xi_0 = \epsilon_1$ and $\nu(X_0)= \epsilon_{n+1}$. Denote $x_0 = (x_1^0, \cdots, x_n^0)$.
As in \cite{CNS5}, we can represent the hypersurface near $X_0$ by the tangential coordinates
$y_1, \cdots, y_n$ and $v(y)$:
\[
X = \sum_{j=1}^n x^0_j e_j + u(x_0) e_{n+1} + \sum_{j=1}^n y_j \epsilon_j + v(y) \epsilon_{n+1}.
\]
Thus, we have $v(0) = 0$ and $\nabla v (0) = 0$. Using the same notations as in \cite{CNS5}, set
$\omega = (1 + |\nabla v|^2)^{1/2}$.
Then the principal curvature in the $\epsilon_1$ direction is
\[
\kappa_1 = \frac{v_{11}}{(1+v_1^2)\omega}.
\]
Under the new coordinates, we denote the point $X \in M_u$ near $X_0$ by $Y = \sum_{j=1}^n y_j \epsilon_j + v(y) \epsilon_{n+1}$.
The normal in the new coordinates still denoted by $\nu$ is given as
\[
\nu(Y) = - \frac{1}{\omega} \sum_{j=1}^n v_j \epsilon_j + \frac{1}{\omega} \epsilon_{n+1}.
\]
Recall that $\epsilon_{n+1} = \frac{1}{w(x_0)} (-u_1(x_0), \cdots, -u_n(x_0), 1)$.
Then,
\begin{equation}
\label{1w}
\frac{1}{w} = \nu(Y) \cdot e_{n+1} = \frac{1}{\omega w(x_0)} - \frac{1}{\omega} \sum_{j=1}^n a_j v_j,
\end{equation}
where $a_j = \epsilon_j \cdot e_{n+1}$. It is easy to see that $\sum_j a_j^2 \leqslant 1$.

We compute that $|Y|^2 = \sum_j y_j^2 + v(y)^2$ and
$|X|^2 = |X_0|^2 + |Y|^2 + 2 \langle X_0, Y\rangle$.
At the point $y = 0$ the function
\begin{equation}
G (y) =  (1/w - a)^{-1 } e^{\frac{A}{2} |X|^2} \frac{v_{11}}{(1+v_1^2)\omega}
\end{equation}
attains its maximum that is equal to $G(x_0, \xi_0)$. At this point,
we also have $v_{1j}= 0$ for $j = 2, \cdots, n$
since the $y_1$ axe lines along the principal direction $\epsilon_1$.
By a rotation of coordinates $y_2, \cdots, y_n$, we may assume that
$v_{ij} (0)$ is diagonal.
By \eqref{matrix} and $\nabla v (0) = 0$, we find that
$a_{il} = v_{il} = \delta_{il} \kappa_i$ at the origin.
This implies that $\tF^{ij}$ is diagonal.
Without loss of generality, we assume that
\[
 \kappa_1 \geqslant \kappa_2 \geqslant \cdots \geqslant \kappa_n.
\]
At the origin, by direct calculations (see (2.12) in \cite{CNS5}), we see that
$\frac{\partial a_{il}}{\partial y_j} = a_{ilj} = v_{ilj}$
and
\[
\frac{\partial^2 a_{il}}{\partial y_1^2} = a_{il11} = v_{il 11} - v_{11}^2 (v_{il} + \delta_{i1} v_{1l} + \delta_{1l} v_{1i}).
\]

Recall that $v(0) = 0$ and $\nabla v (0) = 0$.
%At the origin, we differentiate $\omega$ twice to find that
%\[
%\omega_i = \frac{\sum_k v_k v_{ki}}{\omega} = 0 \; \mbox{and} \; \omega_{ii} = v_{ii}^2, \;\; \mbox{for} \;\; 1\leqslant i \leqslant n.
%\]
Differentiating $\log G(y)$ at $y=0$ twice yields that,
for $1\leqslant i \leqslant n$,
\begin{equation}
\label{diff1}
0 = - \frac{  ( 1/w )_i  }{1/w - a} + A (y_i + \langle X_0, \epsilon_i \rangle)
+ \frac{ v_{11i}}{v_{11}} - \frac{2v_1 v_{1i}}{1+v_1^2} - \frac{\omega_i}{\omega}
\end{equation}
and
\begin{equation}
\label{diff2}
%\frac{U_{11ii}}{U_{11}\log U_{11}} -\frac{U_{11i}^2}{U_{11}^2\log U_{11}} - \frac{U_{11i}^2}{(U_{11}\log U_{11})^2}+ \phi_{ii} \leqslant 0.
0\geqslant - \Big( \frac{  ( 1/w )_i  }{1/w - a} \Big)_i + A + \frac{v_{11ii}}{v_{11}} - \Big(\frac{ v_{11i}}{v_{11}}\Big)^2
- 2v_{1i}^2 - v_{ii}^2.
\end{equation}
Contracting \eqref{diff2} with $\tF^{ii}$, we get
\begin{equation}
\label{S2-1}
\begin{aligned}
0\geqslant &\; \frac{\tF^{ii} v_{11ii}}{v_{11}} - \tF^{ii} \Big(\frac{ v_{11i}}{v_{11}}\Big)^2 - \frac{ \tF^{ii} (1/w)_{ii} }{1/w-a} \\
&\; + \tF^{ii} \Big(\frac{ (1/w)_{i} }{1/w-a}\Big)^2 - 2 \tF^{11} v_{11}^2 - \tF^{ii} v_{ii}^2 + A \sum \tF^{ii}.
\end{aligned}
\end{equation}
By \eqref{1w}, differentiate $1/w$ with respect to $y_i$ to obtain
\[
(1/w)_i = - a_i v_{ii} \;\mbox{and} \; (1/w)_{ii} = -\sum_j a_j v_{jii} - \frac{v_{ii}^2}{w(x_0)}
\]
for all $1\leqslant i \leqslant n$.
Combining the above formulas with the derivatives of $a_{il}$,
we therefore have from \eqref{S2-1} that
\begin{equation}
\label{S2-2}
\begin{aligned}
0\geqslant &\; \frac{\tF^{ii} a_{ii11}}{v_{11}} + v_{11} \tF^{ii} v_{ii} - \tF^{ii} \Big(\frac{ v_{11i}}{v_{11}}\Big)^2
+ \frac{a_j \tilde F^{ii} a_{iij}}{1/w -a}\\
&\; + \frac{\tF^{ii} v_{ii}^2}{1 - a w}
+ \tF^{ii} \Big(\frac{ (1/w)_{i} }{1/w-a}\Big)^2 - \tF^{ii} v_{ii}^2 + A \sum \tF^{ii}.
\end{aligned}
\end{equation}

By differentiating the equation
$\tilde F (a_{ij}) = \tilde f(X,\nu(X)) = \tilde f(Y, \nu(Y))$
twice, we obtain
\[
\tF^{ii} a_{iij} = (\tilde f)_j \geqslant - C -C v_{11}
\]
and
\begin{equation}
\label{diff-eqn-2}
\begin{aligned}
\tF^{ii} a_{ii11} = &\;- \tF^{ij,kl} a_{ij1} a_{kl1} + (\tilde f)_{11} \\
\geqslant &\; - \tF^{ij,kl} a_{ij1} a_{kl1} - C v_{11}^2 - \sum_j \tilde f_{\nu_j} \frac{v_{j11}}{\omega}
\end{aligned}
\end{equation}
for sufficiently large $v_{11}$, where $C$ depends on $|f|_{C^2}$ and $|u|_{C^1}$.
%Also, note that $1 - a w \leqslant \frac{1}{2}$.
Substituting the above inequalities into \eqref{S2-2}, we arrive at
\begin{equation}
\label{S2-3}
\begin{aligned}
0 \geqslant &\;  - \frac{2}{v_{11}} \sum_{i\geqslant 2} \tilde F^{1i,i1} a_{1i1} a_{i11} - Cv_{11}
- \sum_j \tilde f_{\nu_j} \frac{v_{j11}}{v_{11}} - C\\
&\; - \tF^{ii} \Big(\frac{ v_{11i}}{v_{11}}\Big)^2 + \tF^{ii} \Big(\frac{ (1/w)_{i} }{1/w-a}\Big)^2 + \frac{aw}{1- aw} \tF^{ii} v_{ii}^2
+ A \sum \tF^{ii},
\end{aligned}
\end{equation}
provided that $v_{11}$ is large enough, where we used by Lemma \ref{Dinew1} (3) that $\tF^{ii} a_{ii} = \tilde f$.
Using \eqref{diff1}, we further derive that
\begin{equation}
\label{S2-4}
\begin{aligned}
0 \geqslant &\;  \frac{aw}{1- aw} \tF^{ii} v_{ii}^2 - \frac{2}{v_{11}} \sum_{i\geqslant 2} \tilde F^{1i,i1} a_{1i1} a_{i11} - CA\\
&\; - \tF^{ii} \Big(\frac{ v_{11i}}{v_{11}}\Big)^2 + \tF^{ii} \Big(\frac{ (1/w)_{i} }{1/w-a}\Big)^2
+ A \sum \tF^{ii} - C v_{11}
\end{aligned}
\end{equation}
for large $A$.

We now begin to deal with the third order terms.
The following lemma ensures us that in what follows we can assume that
$\kappa_n \geqslant - \delta \kappa_1$, where $\delta \leqslant \frac{1}{2(p-1)}$ is a small positive constant to be determined.
\begin{lemma}
\label{lem3}
Suppose that $\kappa_n \leqslant - \delta \kappa_1$.
Then, we can derive an upper bound for $\kappa_1$ from \eqref{S2-4}.
\end{lemma}

\begin{proof}
By the critical equation \eqref{diff1} and the Cauchy-Schwarz inequality, we have
\[
\tF^{ii} \Big(\frac{ v_{11i}}{v_{11}}\Big)^2 \leqslant  (1+ \varepsilon ) \tF^{ii} \Big(\frac{ (1/w)_{i} }{1/w-a}\Big)^2
+ \Big( 1 + \frac{1}{\varepsilon}\Big) A^2 \tF^{ii} (y_i + \langle X_0, \epsilon_i\rangle )^2
\]
for any $\varepsilon > 0$. From \eqref{S2-4} and $- \tF^{1i,i1} \geqslant 0$, we see that
\begin{equation}
\label{S2-6}
\begin{aligned}
0 \geqslant &\; \frac{aw}{1- aw} \tF^{ii}v_{ii}^2 - \varepsilon \tF^{ii} \Big(\frac{ (1/w)_{i} }{1/w-a}\Big)^2
- \frac{CA^2}{\varepsilon} \sum \tF^{ii} - C v_{11}
\end{aligned}
\end{equation}
for $v_{11}$ sufficiently large.
Using $(1/w)_i = - a_i v_{ii}$ and choosing $\varepsilon$ sufficiently small, we obtain from \eqref{S2-6} that
\[\begin{aligned}
0 \geqslant &\;  \frac{a}{2} \tF^{ii} v_{ii}^2 - \frac{CA^2}{\varepsilon} \sum \tF^{ii} - C v_{11}\\
\geqslant &\; \frac{a\theta}{2} \big(\sum \tF^{ii}\big) \delta^2 \kappa_1^2   - \frac{CA^2}{\varepsilon} \sum \tF^{ii} - Cv_{11},
\end{aligned}\]
where we used Lemma \ref{Dinew1} (4) in the second inequality.
This implies an upper bound for $\kappa_1$ since $\sum \tF^{ii} \geqslant p$.

\end{proof}

For an $\epsilon > 0$ sufficiently small, we choose $\delta = \delta (\epsilon)$ such that Lemma \ref{lem4} holds.
Now we assume $\kappa_n \geqslant - \delta \kappa_1$.
By Lemma \ref{key-1}, the inequality \eqref{S2-4} becomes
\begin{equation}
\label{S2-5}
\begin{aligned}
0 \geqslant &\; \frac{aw}{1- aw} \tF^{ii} v_{ii}^2 - \frac{2}{v_{11}} \sum_{i\geqslant 2} \tilde F^{1i,i1} v_{1i1} v_{i11} - CA\\
&\; - \tF^{ii} \Big(\frac{ v_{11i}}{v_{11}}\Big)^2
 + \tF^{ii} \Big(\frac{ (1/w)_{i} }{1/w-a}\Big)^2
+ \frac{A}{2} \sum \tF^{ii}
\end{aligned}
\end{equation}
as long as $A$ is sufficiently large.
For any given $\epsilon > 0$, by Lemma \ref{lem4}, \eqref{S2-5} becomes
\begin{equation}
\label{S2-5'}
\begin{aligned}
0 \geqslant &\, a \tF^{ii} v_{ii}^2 + \sum_{i\geqslant 2} \tF^{ii} \Big(\frac{ v_{11i}}{v_{11}}\Big)^2
- \tF^{ii} \Big(\frac{ v_{11i}}{v_{11}}\Big)^2 - CA\\
&\; - (1+\epsilon) \sum_{i\geqslant 2} \tF^{11} \Big(\frac{ v_{11i}}{v_{11}}\Big)^2
+ \tF^{ii} \Big(\frac{ (1/w)_{i} }{1/w-a}\Big)^2
+ \frac{A}{2} \sum \tF^{ii}.
\end{aligned}
\end{equation}
By the critical equation \eqref{diff1} and the Cauchy-Schwarz inequality, we see that
\begin{equation}
\label{h111}
\Big(\frac{ v_{111}}{v_{11}}\Big)^2
\leqslant ( 1 + \epsilon ) \Big(\frac{ (1/w)_{1} }{1/w-a}\Big)^2
+ \frac{CA^2 }{\epsilon}
\end{equation}
and, for $i \geqslant 2$,
\begin{equation}
\label{h11i}
(1+\epsilon) \Big(\frac{ v_{11i}}{v_{11}}\Big)^2 \leqslant
 (1+\epsilon)^2 \Big(\frac{ (1/w)_{i} }{1/w-a}\Big)^2 + \frac{CA^2}{\epsilon}.
\end{equation}
%Since $\kappa_n \geqslant - \delta \kappa_1$, we have $F^{11} \leqslant \frac{C}{\kappa_1}$.
%Hence, by Lemma \ref{key-1}, we obtain
%\begin{equation}
%\label{h11i}
%2 \sum_{i\geqslant 2} \tF^{ii} \Big(\frac{ v_{11i}}{v_{11}}\Big)^2 \leqslant \frac{A}{4} \sum \tF^{ii} + CA^2 \tF^{11},
%\end{equation}
Substituting \eqref{h111} and \eqref{h11i} into \eqref{S2-5'},
we have
\begin{equation}
\label{S2-5''}
\begin{aligned}
0 \geqslant &\, a \tF^{ii} v_{ii}^2 - (1 + \epsilon) \tF^{11} \Big(\frac{ (1/w)_{1} }{1/w-a}\Big)^2
  - \frac{CA^2 }{\epsilon} \tF^{11} - CA\\
&\; - (1 + 3\epsilon) \tF^{11} \sum_{i\geqslant 2} \Big(\frac{ (1/w)_{i} }{1/w-a}\Big)^2
+ \tF^{ii} \Big(\frac{ (1/w)_{i} }{1/w-a}\Big)^2 + \frac{A}{2} \sum \tF^{ii}.
\end{aligned}
\end{equation}
Observe that $|v_{ii}| \leqslant n v_{11}$ and $\tF^{11} \leqslant \cdots \leqslant \tF^{nn}$.
We further derive that
\begin{equation}
\label{S2-9}
\begin{aligned}
0 \geqslant &\, a \tF^{ii} v_{ii}^2 - \frac{3n^2 \epsilon}{a^2} \tF^{11} v_{11}^2
  - \frac{CA^2 }{\epsilon} \tF^{11} - CA.
\end{aligned}
\end{equation}
Choose $\epsilon \leqslant \frac{a^3}{6n^2}$.
We then derive an upper bound for $v_{11}$ by Lemma \ref{Dinew1} (1).

\end{proof}

\subsection{Interior $C^2$ estimates}

As in Sheng-Urbas-Wang \cite{SUW}, we actually can derive an interior curvature estimate for equation \eqref{eqn} with $p\geqslant \frac{n}{2}$
by the same argument as that of Theorem 1.3 in \cite{Dong}.
Furthermore, if $f$ does not depend on $\nu$, the interior estimate holds for
all $p$.

\begin{theorem}
\label{int-C2}
Let $\Omega$ be a bounded domain in $\mathbb{R}^n$.
Suppose that $u\in C^4 (\Omega) \cap C^2 (\bar \Omega)$ is an admissible solution to \eqref{eqn} and \eqref{eqn-b}.
Then there is a positive constant $C$ and $\beta$ such that
\begin{equation}
\label{C2-interior}
\sup_{ \Omega} (-u)^\beta|h| \leqslant C,
\end{equation}
where $C$ and $\beta$ depend on $n$, $p$, $|u|_{C^1}$, $\inf \tilde f$ and $|f|_{C^2}$.
\end{theorem}

\begin{proof}
We consider the test function
$G (x, \xi) =  (-u)^\beta (\frac{1}{w} - a)^{-1 } e^{\frac{A}{2} |X|^2} h_{\xi\xi}$.
Assume its maximum is attained at a point $X_0 = (x_0, u(x_0))\in M_u$ and $\xi_0 \in T_{X_0} M_u$.
As before, after choosing new coordinates at $X_0$, the maximum of
\[
G (y) =  (-u)^\beta (\frac{1}{w} - a)^{-1 } e^{\frac{A}{2} |X|^2} \frac{v_{11}}{(1+v_1^2)\omega}
\]
is achieved at $y = 0$ and $a_{il} = v_{il} = \kappa_i \delta_{il}$.
Without loss of generality, we also assume that
$ \kappa_1 \geqslant \kappa_2 \geqslant \cdots \geqslant \kappa_n$; then $\tF^{11} \leqslant \tF^{22} \leqslant \cdots \leqslant \tF^{nn}$.
 Note that
$u = X \cdot e_{n+1} = X_0 \cdot e_{n+1} + \sum_{j=1}^n a_j y_j + \frac{v(y)}{w(x_0)}$,
where $a_j$ are defined as before.
Differentiate $u$ with respect to $y_i$ to obtain
\[
u_i = a_i + \frac{v_i}{w(x_0)} \; \mbox{and}\; u_{ii} = \frac{v_{ii}}{w(x_0)}.
\]

Differentiating $\log G(y)$ at $y=0$ twice yields that,
for $1\leqslant i \leqslant n$,
\begin{equation}
\label{diff1-interior}
0 = \frac{\beta u_i}{u} - \frac{  ( 1/w )_i  }{1/w - a} + A (y_i + \langle X_0, \epsilon_i \rangle)
+ \frac{ v_{11i}}{v_{11}}
\end{equation}
and
\begin{equation}
\label{diff2-interior}
0\geqslant \frac{\beta u_{ii}}{u} - \beta \Big(\frac{ u_i}{u}\Big)^2 - \Big( \frac{  ( 1/w )_i  }{1/w - a} \Big)_i + A + \frac{v_{11ii}}{v_{11}} - \Big(\frac{ v_{11i}}{v_{11}}\Big)^2
- 2v_{1i}^2 - v_{ii}^2.
\end{equation}
Contracting \eqref{diff2-interior} with $\tF^{ii}$ and using $\tF^{ii} v_{ii} = \tF^{ii} a_{ii} = \tilde f$
and \eqref{diff1-interior}, similar to \eqref{S2-4}, we will arrive at
\begin{equation}
\label{In-S2-2}
\begin{aligned}
0 \geqslant &\;  \frac{C\beta}{u} - \beta \tF^{ii} \Big(\frac{ u_i}{u}\Big)^2
- \frac{2}{v_{11}} \sum_{i\geqslant 2} \tilde F^{1i,i1} a_{1i1} a_{i11} - C A - C v_{11}\\
&\; - \tF^{ii} \Big(\frac{ v_{11i}}{v_{11}}\Big)^2 + \tF^{ii} \Big(\frac{ (1/w)_{i} }{1/w-a}\Big)^2 + \frac{aw}{1- aw} \tF^{ii} v_{ii}^2
+ A \sum \tF^{ii}
\end{aligned}
\end{equation}
as long as $v_{11}$ large enough.

If $\kappa_n \leqslant - \delta \kappa_1$ for some $\delta > 0$, similar to Lemma \ref{lem3},
by the critical equation \eqref{diff1-interior} and the Cauchy-Schwarz inequality, we have
\begin{equation}
\label{vi11}
 \Big(\frac{ v_{11i}}{v_{11}}\Big)^2 \leqslant  (1+ \varepsilon ) \Big( \frac{  (1/w)_i }{1/w-a} \Big)^2
+ \frac{CA^2}{\varepsilon}
+ \frac{C \beta^2}{\varepsilon} \frac{ 1}{u^2}
\end{equation}
for any $\varepsilon > 0$,
where $C$ depends on $|Du|_{C^0}$.
Substituting the above inequality into \eqref{In-S2-2},
we then have
\begin{equation}
\label{In-S2-3}
\begin{aligned}
0 \geqslant &\; \frac{a}{2} \tF^{ii} v_{ii}^2
+ \frac{C\beta}{u}
  - \frac{C \beta^2}{\varepsilon} \frac{ \sum \tF^{ii}}{u^2}
- \frac{CA^2}{\varepsilon} \sum \tF^{ii} - Cv_{11}
\end{aligned}
\end{equation}
if $\varepsilon > 0$ is sufficiently small and $v_{11} > A$.
Utilizing $\kappa_n \leqslant - \delta \kappa_1$ and Lemma \ref{Dinew1} (4), we then obtain
\begin{equation}
\label{In-S2-4}
\begin{aligned}
0 \geqslant \frac{a\delta^2 \theta}{2} v_{11}^2 \sum \tF^{ii} + \frac{C\beta}{u}
- \frac{C \beta^2}{\varepsilon} \frac{ \sum \tF^{ii}}{u^2}
- \frac{CA^2}{\varepsilon} \sum \tF^{ii} - C v_{11},
\end{aligned}
\end{equation}
By Lemma \ref{Dinew1} (2), we derive, for some $C$ depending on $n$, $p$, $|u|_{C^1}$ and $|f|_{C^2}$,
that
\[
u^2 v_{11}^2 \leqslant C,
\]
which implies \eqref{C2-interior}.

For any $\epsilon > 0$ sufficiently small, choose $\delta = \delta (\epsilon)$ such that Lemma \ref{lem4} holds.
In the following, we will assume $\kappa_n \geqslant - \delta \kappa_1$.
By Lemma \ref{key-1} and \ref{lem4}, the inequality \eqref{In-S2-2} becomes
\begin{equation}
\label{In-S2-5}
\begin{aligned}
0 \geqslant &\; \frac{C\beta}{u} - \beta \tF^{ii} \Big(\frac{ u_i}{u}\Big)^2
 + (1 + \epsilon) \sum_{i\geqslant 2} \tF^{ii} \Big(\frac{ v_{11i}}{v_{11}}\Big)^2
- \tF^{ii} \Big(\frac{ v_{11i}}{v_{11}}\Big)^2 - CA \\
&\; - (1+\epsilon) \sum_{i\geqslant 2} \tF^{11} \Big(\frac{ v_{11i}}{v_{11}}\Big)^2
 + \tF^{ii} \Big(\frac{ (1/w)_{i} }{1/w-a}\Big)^2 + a \tF^{ii} v_{ii}^2
+ \frac{A}{2} \sum \tF^{ii}
\end{aligned}
\end{equation}
as long as $A$ is sufficiently large.
By \eqref{diff1-interior} and the Cauchy-Schwarz inequality, we have, for $i\geqslant 2$,
\[
- \beta  \tF^{ii} \Big(\frac{ u_i}{u}\Big)^2 \geqslant
- \frac{3}{\beta}  \tF^{ii} \Big( \frac{ v_{11i} }{v_{11}}\Big)^2
- \frac{3}{\beta}  \tF^{ii} \Big(\frac{ (1/w)_{i} }{1/w-a}\Big)^2
- \frac{C A^2}{\beta} \tF^{ii}.
\]
Taking $i = 1$ and $\varepsilon = \epsilon$ in \eqref{vi11} and substituting it together with the
above inequality into \eqref{In-S2-5}, we obtain
\begin{equation}
\label{In-S2-6}
\begin{aligned}
0 \geqslant &\;  a \tF^{ii} v_{ii}^2  - \beta \tF^{11} \Big(\frac{u_1}{u}\Big)^2
 + (\epsilon - \frac{3}{\beta}) \sum_{i\geqslant 2} \tF^{ii} \Big(\frac{ v_{11i}}{v_{11}}\Big)^2
 - \frac{C A^2}{\beta} \sum \tF^{ii}\\
&\; - \epsilon \tF^{11} \Big(\frac{ (1/w)_{1} }{1/w-a}\Big)^2
 -\frac{CA^2}{\epsilon} \tF^{11} - \frac{C \beta^2}{\epsilon} \frac{ \tF^{11}}{u^2}
 + \frac{A}{2} \sum \tF^{ii} - CA\\
&\; - (1+\epsilon) \sum_{i\geqslant 2} \tF^{11} \Big(\frac{ v_{11i}}{v_{11}}\Big)^2
+ (1 - \frac{3}{\beta}) \sum_{i\geqslant 2} \tF^{ii} \Big(\frac{ (1/w)_{i} }{1/w-a}\Big)^2 + \frac{C\beta}{u}.
\end{aligned}
\end{equation}
Recall that $\tF^{11} \leqslant \tF^{22} \leqslant \cdots \leqslant \tF^{nn}$.
Choosing $\epsilon$ sufficiently small and $\beta$ sufficiently large and using \eqref{vi11} for $i\geqslant 2$, we arrive at
\begin{equation}
\label{In-S2-7}
\begin{aligned}
0 \geqslant &\;  \frac{a}{2} \tF^{ii} v_{ii}^2
- (3\epsilon + \frac{3}{\beta}) \sum_{i\geqslant 2} \tF^{ii} \Big(\frac{ (1/w)_{i} }{1/w-a}\Big)^2\\
&\; + \frac{C\beta}{u} - \frac{C \beta^2}{\epsilon} \frac{ \tF^{11}}{u^2} -\frac{CA^2}{\epsilon} \tF^{11}
+ \frac{A}{4} \sum \tF^{ii} - CA\\
\geqslant &\;  \frac{a}{4} \tF^{ii} v_{ii}^2
+ \frac{C\beta}{u} - \frac{C \beta^2}{\epsilon} \frac{ \tF^{11}}{u^2} -\frac{CA^2}{\epsilon} \tF^{11}  - CA.
\end{aligned}
\end{equation}
By Lemma \ref{Dinew1} $(1)$, we conclude
\[
u^2 v_{11}^2 \leq C
\]
from the last inequality of \eqref{In-S2-7}, where $C$ depends on
$n$, $p$, $|u|_{C^1}$, $\inf \tilde f$ and $|f|_{C^2}$.
This completes the proof of Theorem \ref{int-C2}.

\end{proof}

\section{Boundary $C^2$ estimates}

In this section, we establish the boundary $C^2$ estimate. The main theorem in this section is stated as below.
\begin{theorem}
Let $\Omega$ be a strictly convex bounded domain in $\mathbb{R}^n$ with smooth boundary $\partial \Omega$.
Suppose $u\in C^3(\bar \Omega)$ is an admissible solution to \eqref{eqn} and \eqref{eqn-b}.
Then, there exists a positive constant $C$ such that
\begin{equation}
\label{C2b}
\max_{\partial \Omega} |D^2 u| \leqslant C,
\end{equation}
where $C$ depends on $n,p$, $|u|_{C^1}$, $|f|_{C^1}$, $\inf f$ and $\partial \Omega$.
\end{theorem}

Without loss of generality, we assume $0 \in \partial \Omega$.
Suppose that the boundary $\partial \Omega$ around the origin is given by
\begin{equation}
\label{boundary}
x_n = \rho(x') = \frac{1}{2} \sum_{\alpha < n} \kappa_\alpha^b x_\alpha^2 + O(|x'|^3),
\end{equation}
where $x' = (x_1, \cdots, x_{n-1})$ and $\kappa_1^b, \cdots, \kappa_{n-1}^b$ are the principal curvatures of $\partial \Omega$ at the origin.
Differentiating the boundary condition $u = 0$ at the origin twice, we obtain that
\begin{equation}
\label{C2tangential}
|u_{\alpha\beta} (0)| \leqslant C, \; \mbox{for} \; 1\leqslant \alpha, \beta \leqslant n-1.
\end{equation}

For the tangential-normal estimate, we need some idea from \cite{I1}.
Let $d(x) = dist (x, \partial \Omega)$ be the distance function to the boundary and $\Omega_\delta := \{x\in\Omega: |x| < \delta\}$.
We consider
\[
W:= u_{\alpha} + \rho_{\alpha} u_n - \frac{1}{2} \sum_{1\leqslant \beta \leqslant n-1} u_{\beta}^2,
\]
where $1\leqslant \alpha\leqslant n-1$. Similar to Lemma 5.1 in \cite{I1} and Lemma 5.3 in \cite{JS}, we can prove the following lemma.
\begin{lemma}
\label{Ivo}
For sufficiently small $\delta> 0$, we have,
\begin{equation}
\label{Ivo-inequality}
G^{ij} W_{ij} \leqslant C (1 + |DW| + \sum_i G^{ii} + G^{ij} W_i W_j)\; \mbox{in}\; \Omega_\delta,
\end{equation}
where $C$ depends on $n$, $p$, $|u|_{C^1}$, $|\tilde f|_{C^1}$, $\inf \tilde f$ and $\partial \Omega$.
\end{lemma}

\begin{proof}
Note that
\begin{equation}
\label{GW}
\begin{aligned}
G^{ij}W_{ij} + G^s W_s = &\; \tilde f_\alpha + \rho_\alpha \tilde f_n - \sum_{\beta\leqslant n-1} u_{\beta} \tilde f_{\beta}
+ 2 G^{ij} u_{ni}\rho_{\alpha j}\\
&\; - \sum_{\beta\leqslant n-1}G^{ij} u_{\beta i} u_{\beta j} + u_n G^{ij}\rho_{\alpha ij} + u_n G^{s}\rho_{\alpha s}.
\end{aligned}
\end{equation}
It is easy to see that
\begin{equation}
\label{Gij}
G^{ij} = \frac{1}{w} \sum_{s,t} \tF^{st} \gamma^{si}\gamma^{jt}\;
\mbox{and}\;
u_{ij} = w \sum_{s,t} \gamma_{is} a_{st} \gamma_{tj}.
\end{equation}
It follows that
\[
\sum_{\beta\leqslant n-1}G^{ij} u_{\beta i} u_{\beta j}
= w \sum_{\beta \leqslant n-1} \sum_{s,t} \frac{\partial \tF}{\partial a_{ij}} \gamma_{\beta s} \gamma_{\beta t} a_{si} a_{tj}.
\]
Suppose $b_{ij}$ is the orthogonal matrix that diagonalize $a_{ij}$ and $\tF^{ij}$
simultaneously, i.e.
\begin{equation}
\label{Fij}
\tF^{ij} = \sum_s b_{is} \frac{\partial \tF}{\partial \kappa_s} b_{js} \;
\mbox{and}\;
a_{ij} = \sum_s b_{is} \kappa_s b_{js}.
\end{equation}
Therefore, we get
\[
\sum_{\beta\leqslant n-1}G^{ij} u_{\beta i} u_{\beta j}
= w\sum_{\beta \leqslant n-1} \sum_{i} \Big(\sum_s \gamma_{\beta s} b_{si}\Big)^2
\frac{\partial \tF}{\partial \kappa_i} \kappa_i^2.
\]
Define the matrix $\eta = (\eta_{ij})$ by
\[
\eta_{ij} = \sum_s \gamma_{is} b_{sj}.
\]
It is easy to verify that $\eta \cdot \eta^T = g$ and $|\det(\eta)| = \sqrt{1 + |Du|^2}$. Hence,
\begin{equation}
\label{G-f}
\sum_{\beta\leqslant n-1}G^{ij} u_{\beta i} u_{\beta j}
= w\sum_{\beta \leqslant n-1} \sum_{i} (\eta_{\beta i})^2
\frac{\partial \tF}{\partial \kappa_i} \kappa_i^2.
\end{equation}
By \eqref{Gij} and \eqref{Fij}, we obtain
\begin{equation}
\label{Fa}
 |G^{ij} u_{ni} \rho_{\alpha j}| \leqslant C \sum_i
\frac{\partial \tF}{\partial \kappa_i} |\kappa_i | \; \mbox{and}\;
\tF^{ij}a_{it} = \sum_{i} \frac{\partial \tF}{\partial \kappa_i} \kappa_i b_{ji}b_{ti}.
\end{equation}
By \eqref{GS-equality} and the above equality, we derive that
\[
G^s \rho_{\alpha s} \leqslant C \sum_i \frac{\partial \tF}{\partial \kappa_i} |\kappa_i |.
\]
A direct calculation shows that
\[
|\tilde f_{\alpha} + \tilde f_n\rho_\alpha - \sum_{\beta \leqslant n-1} u_\beta \tilde f_\beta| \leqslant C (1 + |DW|).
\]
We therefore arrive at
\begin{equation}
\label{G-W}
\begin{aligned}
G^{ij}W_{ij} \leqslant &\;C \Big(1 + |DW|+ \sum G^{ii}  \Big) - G^s W_s \\
&\; + C \sum_i \frac{\partial \tF}{\partial \kappa_i} |\kappa_i | - \sum_{\beta\leqslant n-1}G^{ij} u_{\beta i} u_{\beta j} .
\end{aligned}
\end{equation}
%Since $\sum_i \frac{\partial \tF}{\partial \kappa_i} \kappa_i = \tilde f$, we further have
%\begin{equation}
%\label{G-W'}
%\begin{aligned}
%G^{ij}W_{ij} \leqslant &\;C \Big(1 + |DW|+ \sum G^{ii}   \Big) - G^s W_s\\
%&\; - 2 C\sum_{\kappa_i < 0} \frac{\partial \tF}{\partial \kappa_i} \kappa_i - w\sum_{\beta \leqslant n-1} \sum_{i} \eta_{\beta i}^2
%\frac{\partial \tF}{\partial \kappa_i} \kappa_i^2  .
%\end{aligned}
%\end{equation}

Next we estimate $G^s W_s$. By Lemma \ref{GS} and \eqref{Fa}, we see that
\begin{equation}
\label{GW-1}
\begin{aligned}
-G^s W_s
=&\; \frac{u_sW_s}{w^2} \tilde f
+ \frac{2}{w} \sum_{t,i} \frac{\partial \tF}{\partial \kappa_i} \kappa_i (b_{ti}u_t)(b_{ji} \gamma^{sj}) W_s\\
\leqslant &\; C |DW| + \frac{2}{w} \sum_{t,i} \frac{\partial \tF}{\partial \kappa_i} \kappa_i (b_{ti}u_t)(b_{ji} \gamma^{sj}) W_s.
\end{aligned}
\end{equation}
Now we divide the proof into two cases: (a) $ \sum_{\beta \leqslant n-1} \eta_{\beta i}^2 \geqslant \epsilon$
for all index $i$; (b) $ \sum_{\beta \leqslant n-1} \eta_{\beta r}^2 < \epsilon$ for some index $1\leqslant r \leqslant n$,
where $\epsilon > 0$ is a small constant to be determined later.

\textbf{Case (a)}. By our assumption and \eqref{G-f}, we have
\begin{equation}
\label{Guu}
\sum_{\beta\leqslant n-1}G^{ij} u_{\beta i} u_{\beta j}
\geqslant \epsilon \sum_i
\frac{\partial \tF}{\partial \kappa_i} \kappa_i^2.
\end{equation}
The second term in the right hand side of the inequality \eqref{GW-1} can be estimated as
\begin{equation}
\label{GW-2}
 \frac{\partial \tF}{\partial \kappa_i} \kappa_i (b_{ti}u_t)(b_{ji} \gamma^{sj}) W_s
\leqslant \frac{\epsilon}{4} \frac{\partial \tF}{\partial \kappa_i} \kappa_i^2
+ \frac{C}{\epsilon} \frac{\partial \tF}{\partial \kappa_i} (b_{ji} \gamma^{sj} W_s)^2
\end{equation}
by the Cauchy-Schwarz inequality. Note that
\begin{equation}
\label{GW-3}
\frac{\partial \tF}{\partial \kappa_i} (b_{ji} \gamma^{sj} W_s)^2 = G^{ij} W_i W_j.
\end{equation}
Combining \eqref{GW-3}, \eqref{GW-2} with \eqref{GW-1}, we obtain
\begin{equation}
\label{GW-4}
-G^s W_s \leqslant C (|DW|) + \frac{\epsilon}{2} \sum_i \frac{\partial \tF}{\partial \kappa_i} \kappa_i^2
+ \frac{C}{\epsilon} G^{ij} W_i W_j.
\end{equation}
Substituting \eqref{GW-4} and \eqref{Guu} into \eqref{G-W} and applying the Cauchy-Schwarz inequality to $\sum \frac{\partial \tF}{\partial \kappa_i} |\kappa_i|$, we get the desired
inequality \eqref{Ivo-inequality}.

\textbf{Case (b)}.
We may assume that the principal curvatures are ordered as $\kappa_1 \geqslant \kappa_2 \geqslant \cdots \geqslant \kappa_n$.
Then, we have
$\frac{\partial \tF}{\partial \kappa_1} \leqslant \frac{\partial \tF}{\partial \kappa_2} \leqslant \cdots \leqslant \frac{\partial \tF}{\partial \kappa_n}$.
Taking into account the fact that $|\eta_{nr}| \leqslant \sqrt{1 + |Du|_{C^0}^2}$,
we see that
\[
1 \leqslant |\det(\eta)| \leqslant \sqrt{1 + |Du|_{C^0}^2} | \det(\eta')| + C\epsilon,
\]
where $\eta' = \{\eta_{pq}\}_{p\neq n, q\neq r}$, which implies
\[
 | \det(\eta')| \geqslant \frac{1}{2\sqrt{1 + |Du|_{C^0}^2}}
\]
for sufficiently small $\epsilon$. On the other hand, for any fixed index $i \neq r$, it holds that
\[
|\det (\eta')| \leqslant C \sum_{\beta \neq n}|\eta_{\beta i}|
\]
for some positive constant $C$ depending on $n$ and $|Du|_{C^0}$. By the above two inequalities, we derive that,
for any $i\neq r$,
\[
\sum_{\beta \leqslant n-1} \eta_{\beta i}^2 \geqslant c_1
\]
for some positive constant $c_1$ depending on $n$ and $|Du|_{C^0}$.
Substituting the above inequality into \eqref{G-f}, we obtain
\begin{equation}
\label{G-f'}
\sum_{\beta\leqslant n-1}G^{ij} u_{\beta i} u_{\beta j}
\geqslant c_1 \sum_{i\neq r}
\frac{\partial \tF}{\partial \kappa_i} \kappa_i^2.
\end{equation}
If $\kappa_r \leqslant 0$, by the same argument as that of Lemma 2.20 in \cite{G} and
$\sum \frac{\partial \tF}{\partial \kappa_i} \kappa_i \geqslant 0$, we obtain
\[
\sum_{i\neq r}
\frac{\partial \tF}{\partial \kappa_i} \kappa_i^2 \geqslant \frac{1}{n} \sum_{i}
\frac{\partial \tF}{\partial \kappa_i} \kappa_i^2.
\]
We then get the desired
inequality \eqref{Ivo-inequality} by the same argument as Case (a) with $\epsilon$ in \eqref{Guu} replaced by $\frac{c_1}{n}$.
Therefore, we may assume $\kappa_r >0$.
Note that $\kappa_r \geqslant \kappa_n$.
It is easy to see that
\[
(b_{jr} \gamma^{sj}) W_s = w \Big(\eta_{\alpha r} + \rho_{\alpha} \eta_{nr} - \sum_{\beta\leq n-1} u_\beta \eta_{\beta r} \Big)\kappa_r
+ b_{jr}\gamma^{sj} \rho_{\alpha s} u_n,
\]
which implies that
\begin{equation}
\label{Ws}
(b_{jr} \gamma^{sj}) W_s \leq Cw (\epsilon + |\rho_\alpha|)\kappa_r + C.
\end{equation}

First, suppose that $\kappa_n \geqslant - \frac{\kappa_r}{2(p-1)}$.
%We further consider two cases: $\kappa_r \geqslant (\varepsilon)^{\frac{1}{2}} \kappa_1$
%and $\kappa_r \leqslant (\varepsilon)^{\frac{1}{2}} \kappa_1$.
%If $\kappa_r \geqslant (\varepsilon)^{\frac{1}{2}} \kappa_1$ holds,
We see
$\kappa_r + \kappa_{i_1} + \cdots + \kappa_{i_{p-1}} \geqslant \frac{\kappa_r}{2}$. It follows that
\begin{equation}
\label{F-r}
\begin{aligned}
\frac{\partial F}{\partial \kappa_r}
= &\; \sum_{r \notin\{i_1, \cdots, i_{p-1}\}} \frac{F(\kappa) }{\kappa_r + \kappa_{i_1} + \cdots + \kappa_{p-1}}
\leqslant \frac{C  }{ \kappa_r }.
\end{aligned}
\end{equation}
By \eqref{F-r} and the Cauchy-Schwarz inequality, we have
\[\begin{aligned}
&\; \sum_{t,i} \frac{\partial \tF}{\partial \kappa_i} \kappa_i (b_{ti}u_t)(b_{ji} \gamma^{sj}) W_s
\leqslant C |DW|
+ \frac{\epsilon}{4} \sum_{i\neq r} \frac{\partial \tF}{\partial \kappa_i} \kappa_i^2
+ \frac{C}{\epsilon} G^{ij} W_i W_j.
\end{aligned}\]
Substituting the above inequality into \eqref{GW-1}, we obtain
\begin{equation}
\label{GW-5}
-G^s W_s
\leqslant C |DW| + \frac{\epsilon}{4} \sum_{i\neq r} \frac{\partial \tF}{\partial \kappa_i} \kappa_i^2
+ \frac{C}{\epsilon} G^{ij} W_i W_j.
\end{equation}
By \eqref{GW-5}, \eqref{F-r}, \eqref{G-f'} and applying the Cauchy-Schwarz inequality to
$\sum_{i\neq r} \frac{\partial \tF}{\partial \kappa_i} |\kappa_i|$,
we conclude \eqref{Ivo-inequality} from \eqref{G-W}.

Now suppose that $\kappa_n \leqslant - \frac{\kappa_r}{2(p-1)}$.
Note that $r\neq n$ in this case.
Since $\frac{\partial \tF}{\partial \kappa_r} \kappa_r = \tilde f - \sum_{i\neq r} \frac{\partial \tF}{\partial \kappa_i} \kappa_i$,
we have, by \eqref{Ws} and the Cauchy-Schwarz inequality,
\[\begin{aligned}
&\; \sum_{t,i} \frac{\partial \tF}{\partial \kappa_i} \kappa_i (b_{ti}u_t)(b_{ji} \gamma^{sj}) W_s \\
%\leqslant &\;
%C \frac{\partial \tF}{\partial \kappa_r} \kappa_r (b_{jr} \gamma^{sj}) W_s +
%\frac{\epsilon}{4} \sum_{i\neq r} \frac{\partial \tF}{\partial \kappa_i} \kappa_i^2
%+ \frac{C}{\epsilon} G^{ij} W_i W_j\\
\leqslant &\; C \frac{\partial \tF}{\partial \kappa_r} \kappa_r +  C (\epsilon + |\rho_\alpha|) \frac{\partial \tF}{\partial \kappa_r} \kappa_r^2
+ \frac{\epsilon}{4} \sum_{i\neq r} \frac{\partial \tF}{\partial \kappa_i} \kappa_i^2
+ \frac{C}{\epsilon} G^{ij} W_i W_j\\
\leqslant &\; C + \frac{C}{\epsilon}\sum_{i\neq r}\frac{\partial \tF}{\partial \kappa_i} + C (\epsilon + |\rho_\alpha|) \frac{\partial \tF}{\partial \kappa_n} \kappa_n^2
+ \frac{\epsilon}{2} \sum_{i\neq r} \frac{\partial \tF}{\partial \kappa_i} \kappa_i^2
+ \frac{C}{\epsilon} G^{ij} W_i W_j.
\end{aligned}\]
Combining \eqref{G-W}, \eqref{GW-1} and \eqref{G-f'} with the above inequality and
 choosing $\epsilon,\delta$ sufficiently small such that $C(\epsilon + \delta) < c_1$, we can get \eqref{Ivo-inequality}.
Thus, the proof of Lemma \ref{Ivo} is finished.

\end{proof}

With the above lemma in hand, we can construct a suitable barrier function to prove the following estimate
\begin{equation}
\label{tangential-normal}
|u_{\alpha n} (0)| \leqslant C, \; \mbox{for} \; 1 \leqslant \alpha \leqslant n-1.
\end{equation}
Since the boundary $\partial \Omega$ is smooth and strictly convex,
we can find a smooth strictly convex function $v$ satisfying $v = 0$ on $\partial \Omega$ and $v \leqslant 0$ in $\Omega$.
Therefore, we have $D^2 v(x) \in P_1$ for $x\in \ol \Omega$.
Moreover, we can find a small positive constant $\theta$ such that $\tilde v = v - \theta |x|^2$
satisfies $D^2 \tilde v (x) \in P_1$ for $x\in \ol \Omega$.
Now we can finish the proof of \eqref{tangential-normal}. Consider the following barrier function
\begin{equation}
\label{barrier}
\Psi = \tilde v + \frac{\theta}{2} |x|^2 - td + \frac{N}{2} d^2,
\end{equation}
where $t$ and $N$ are two positive constants to be chosen later.
%This kind of barrier was used by Guan frequently as in \cite{G, GJ}
A direct calculation shows that
\begin{equation}
\label{Gpsi}
\begin{aligned}
G^{ij}\Psi_{ij} %= &\; G^{ij} (D^2 v + (Nd-t) D^2 d + N Dd\otimes Dd)_{ij}\\
= G^{ij} (D^2 \tilde v + N Dd\otimes Dd)_{ij} + G^{ij} (Nd-t) d_{ij} + \theta \sum G^{ii}.
\end{aligned}
\end{equation}
Concerning $D^2 \tilde v \in P_1$, we see that $D^2 \tilde v + N Dd\otimes Dd \in P_1$.
By the concavity and homogeneity of $G$ with respect to $D^2 u$, we see that
\begin{equation}
\label{GV}
G^{ij} (D^2 \tilde v + N Dd\otimes Dd)_{ij} \geq G( D^2 \tilde v + NDd\otimes Dd, Du) = \tilde F(V_u),
\end{equation}
where $V_u := \frac{1}{w} \{ \gamma^{ik} V_{kl} \gamma^{lj} \}$ and $V:=\{\tilde v_{kl} + N d_k d_l\}$.
We adopt the convention that the eigenvalues are arranged in algebraically nondecreasing order:
\[
 \lambda_1 \leqslant \lambda_2 \leqslant \cdots \leqslant \lambda_n.
\]
Then, by the Weyl theorem, it is easy to see that $\lambda_\alpha (V) \geqslant \lambda_\alpha (D^2 \tilde v)$
for $\alpha = 1, 2, \cdots, n-1$
and $\lambda_n (V) \geqslant \lambda_1 (D^2 \tilde v ) + N$.
Denote $\theta_0 = \inf_{x\in \Omega} \frac{\lambda_1 (D^2 \tilde v)}{\lambda_n (D^2 \tilde v)}$.
We further derive that
$\lambda_\alpha (V) \geqslant \theta_0 \lambda_\alpha (D^2 \tilde v)$ for $\alpha = 1, 2, \cdots, n-1$
and $\lambda_n (V) \geqslant \theta_0 (\lambda_n (D^2 \tilde v ) + N)$.
On the other hand, by the Ostrowski Theorem, there exist $\theta_k > 0$ such that
$\lambda_k (V_u) = \theta_k \lambda_k (V)$ for $k =1,2, \cdots, n$,
where $\theta_k$ is uniformly positive only depending on $|Du|_{C^0}$.
Substituting \eqref{GV} into \eqref{Gpsi} and by Lemma \ref{growth}, we have, for sufficiently large $N$ and sufficiently small $\delta$ and $t$,
\[
G^{ij}\Psi_{ij} \geqslant \bar\theta N^{\frac{p}{n}} + \frac{\theta}{2} \sum G^{ii}
\geqslant N^{\frac{p-1}{n}} + \frac{\theta}{2} \sum G^{ii},
\]
where $\bar\theta$ depends on $\theta_0, \theta_1, \cdots \theta_n$ and $\lambda( D^2 \tilde v)$.
Also, note that
$| D\Psi | = |D \tilde v + \theta x - tDd + N d Dd| \leqslant C + N \delta$
in $\Omega_{\delta}$. Choosing $\delta $ sufficiently small, we obtain
\[
G^{ij}\Psi_{ij} \geqslant N^{\frac{1}{n}} |D\Psi| + \frac{\theta}{2} \Big( \sum G^{ii} + 1\Big).
\]

Let
$\tilde W = 1 - e^{-B W}$, where constant $B>0$ is sufficiently large.
It follows that
\begin{equation}
\label{MP-critical-point}
G^{ij} (R \Psi - \tilde W)_{ij} \geqslant \sqrt[n]{N} (|DR\Psi| - |D \tilde W|)
\end{equation}
if we choose $R \gg B \gg 1$.
Since $\Psi|_{\partial \Omega_\delta \cap \partial \Omega} \leqslant -\frac{\theta}{2}|x'|^2$,
$\Psi|_{\partial \Omega_{\delta} \cap \Omega} \leqslant -\frac{\theta}{2}\delta^2$ and
$-W|_{\partial \Omega} = \frac{1}{2} \sum_\beta u_\beta^2 = \frac{1}{2} \sum_\beta \rho_\beta^2 u_n^2 \leqslant C|x'|^2$
for some $C$ depending on $|\rho|_{C^2}$ and $|u|_{C^1}$,
it is easy to verify that on $\partial \Omega_\delta$ we have
$R\Psi - \tilde W \leqslant 0$ if $R$ is large enough.
By the maximum principle we get
\[
R\Psi - \tilde W \leqslant 0 \;\mbox{in}\; \Omega_\delta.
\]
As $(R\Psi - \tilde W) (0) = 0$ we deduce that
$u_{\alpha n} (0) \geqslant -C$.
Consider $W' = - u_{\alpha} - \rho_{\alpha} u_n - \frac{1}{2} \sum_{1\leqslant \beta \leqslant n-1} u_{\beta}^2$. By the same argument we deduce
$u_{\alpha n} (0) \leqslant C$.
The proof of \eqref{tangential-normal} is completed.

Finally, we prove the double normal estimate
\begin{equation}
\label{doublenormal}
|u_{nn}(0)| \leqslant C.
\end{equation}
We only need to derive an upper bound for $u_{nn}(0)$ by $\sum_i \kappa_i > 0$ since $P_{p} \subset P_{n}$.
Actually, by the inequality of arithmetic and geometric means, we have
$\frac{ p \sum_{i} \kappa_i }{n} \geqslant \tF (\kappa)\geqslant \inf \tilde f > 0$.
By (1.8) in \cite{CNS5}, i.e. $ad(x) \leqslant -u(x) \leqslant \frac{1}{a} d(x)$ in $\Omega$
for some constant $a>0$ depending on $\inf \tilde f$, we obtain
\[
u_n (0) \leqslant -a.
\]
We have, with respect to a principal coordinate system at the origin, $u_{\alpha\beta} (0) = -u_n \kappa_\alpha^b \delta_{\alpha\beta}$ for $1\leqslant \alpha, \beta \leqslant n-1$
and $g^{ij} = \delta_{ij} - \frac{|Du|^2}{w^2} \delta_{in}\delta_{jn}$.
Therefore, the matrix in \eqref{matrix} has the following form
\[
%\frac{1}{w}
\left( \begin{array}{ccccc}
-u_n \kappa_1^b & 0 & \cdots & 0& \frac{u_{1n}}{w} \\
0 & -u_n \kappa_2^b & \cdots & 0& \frac{u_{2n}}{w} \\
\vdots & \vdots &  \ddots & \vdots & \vdots\\
0& 0 & \cdots & -u_n \kappa_{n-1}^b & \frac{u_{n-1 n}}{w} \\
\frac{u_{n1}}{w} & \frac{u_{n2}}{w} & \cdots & \frac{u_{n n-1}}{w} & \frac{u_{nn}}{w^2} \end{array} \right).
\]
By Lemma 1.2 of \cite{CNS3}, when $|u_{nn}|$ goes to infinity, the eigenvalues $w(\kappa_1, \cdots, \kappa_n)$ of the above matrix behave like
\[
w \kappa_\alpha = -u_n\kappa_\alpha^b + o(1),\; 1 \leqslant \alpha\leqslant n-1, \; w\kappa_n = \frac{u_{nn}}{w^2}  + O(1),
\]
where $o(1)$ and $O(1)$ are uniform only depending on $-u_n\kappa^b_\alpha$ and $|u_{n\alpha} (0)|$.
Since $-u_n(0)\geqslant a$ and $(\kappa_1^b, \cdots, \kappa_{n-1}^b) \in \mathcal{P}_{p}$, there exists a uniform positive constant $\varsigma$
such that
\[
\Pi_{1\leqslant i_1 < \cdots < i_{p} \leqslant n-1} (\kappa_{i_1} + \cdots + \kappa_{i_{p}}) \geqslant \varsigma,
\]
as long as $u_{nn}$ is large enough.
Returning to the equation \eqref{eqn}, we have
\[
 \Pi_{1\leqslant i_1 < \cdots < i_{p-1} \leqslant n-1}( \kappa_n + \kappa_{i_1} + \cdots + \kappa_{i_{p-1}}) \leqslant \frac{C}{\varsigma},
\]
from which we can derive an upper bound for $\kappa_n$. Therefore, $u_{nn} (0) $ is bounded from above.
Combining \eqref{C2tangential}, \eqref{tangential-normal} with \eqref{doublenormal}, we obtain \eqref{C2b}.

\section{Gradient Estimates}

In this section, we establish the gradient estimate. First,
by the existence of a subsolution $\underline{u}$ and the maximum principle, it is easy to derive
\begin{equation}
\label{C0-C1b}
\sup_{\Omega} |u| + \sup_{\partial \Omega} |Du| \leqslant C,
\end{equation}
where $C$ depends on $|\underline u|_{C^1}$.
Next, we prove the global gradient estimate following the argument in \cite{CNS5}. We have
\begin{theorem}
Let $u \in C^3(\Omega)\cap C^1 (\bar \Omega)$ be an admissible solution to the problem \eqref{eqn} and \eqref{eqn-b}.
Suppose $f (x,z,\nu) \in C^{\infty} (\bar \Omega \times \mathbb{R} \times \mathbb{S}^n) > 0$ and $f_z \geqslant 0.$
Then, the estimate
\begin{equation}
\label{gradient}
\sup_{\Omega} |Du| \leqslant C (1 + \sup_{\partial \Omega} |Du|)
\end{equation}
holds for a positive constant $C$ depending on $n$, $p$, $|u|_{C^0}$, $\inf f$ and $|f|_{C^1}$.
\end{theorem}

\begin{proof}
%Let $\epsilon_1, \cdots, \epsilon_{n+1}$ be a standard orthogonal basis of $\mathbb{R}^{n+1}$. Consider the following
%quantity
%\[
%Q = \log w + B\langle X, \epsilon_{n+1}\rangle
%\]
%where $B$ is a large positive constant to be determined.
%Suppose the maximum of $Q$ is achieved at $x_0\in \Omega$. We rotate $\epsilon_1, \cdots \epsilon_n$ such that
%$u_1 (x_0) = |Du| (x_0)$. Define $e_i = \gamma^{is} \tilde \partial_s$, where $\tilde \partial_s = \epsilon_s + u_s \epsilon_{n+1}$
%for $1\leqslant s \leqslant n$. Then, $\{e_1, \cdots, e_n\}$ is an orthonormal frame on $M_u$ and, at $X_0 = (x_0, u(x_0))$,
%\[
%\nabla_1 u = \frac{|Du|}{w} = |\nabla u|, \nabla_i u = u_i = 0, \; i\geqslant 2.
%\]
%By further rotation, we may assume $\{u_{ij}\}_{i,j\geqslant 2}$ is diagonal at $x_0$.
%By Gauss formula, we see
%\[
%\nabla_i w = \frac{h_{ik} \nabla_k u}{w^2}.
%\]
%Differentiating $Q$ at the point $x_0$ once, we get
%\begin{equation}
%\label{D1Q}
%\nabla_i Q = w h_{i1} \nabla_1 u + B \nabla_i u = 0,\; i= 1, \cdots, n,
%\end{equation}
%which implies that
%\[
%w h_{11} = - B, \; h_{i1} = 0, \; i = 2,\cdots, n.
%\]

As in \cite{CNS5}, we consider the following quantity
\[
Q = |Du|e^{Au},
\]
where $A$ is a large constant to be determined later.
Suppose that $Q$ achieves its maximum at $x_0 \in \Omega$. At this point we may rotate the coordinates such that
\[
u_1 = |Du| > 0\; \mbox{ and}\; u_{\alpha} = 0 \;\mbox{for}\; \alpha \geqslant 2.
\]
Differentiating $\log Q$ at $x_0$ once, we obtain
\begin{equation}
\label{D1Q}
\frac{u_{1i}}{u_1} + A u_i = 0\; \mbox{for} \; i= 1, \cdots, n.
\end{equation}
By the above equality, we may further rotate the coordinates so that $u_{ij}$ is diagonal at $x_0$.
From \eqref{matrix}, we find that
\[
a_{11} = \frac{u_{11}}{w^3}, \; a_{ii} = \frac{u_{ii}}{w} \; \mbox{for} \; i\geqslant 2 \; \mbox{and}\; a_{ij} = 0 \;\mbox{for}\; i\neq j .
\]
Denote by $\alpha_{ij}$ the positive semidefinite matrix defined as
\[
\alpha_{11} = \frac{1}{w^3}, \; \alpha_{1i} = \frac{1}{w^2},\; i \geqslant 2, \; \alpha_{ij} = \frac{1}{w},\; j\geqslant i \geqslant 2.
\]
As in \cite{Wang}, denote by $\hat F^{ij} = F^{ij}\alpha_{ij}$ the Hadamard product of the matrix $F^{ij}$ and $\alpha_{ij}$, which is positive definite since $F^{ij}$ is diagonal and positive definite.
Differentiating $\log Q$ at $x_0$ a second time and contracting with $\hat F^{ij}$, we get
\begin{equation}
\label{D2Q}
0 \geqslant \frac{\hat F^{ij}u_{1ij}}{u_1} - \frac{\hat F^{ij } u_{1i} u_{1j}}{u_1^2} + A \hat F^{ij} u_{ij}.
\end{equation}
Differentiating $a_{ij}$ and the equation \eqref{eqn}, we obtain
\begin{equation}
\label{da}
F^{ij} a_{ij1} = \hat F^{ij} u_{ij1} - \frac{2F^{11} u_{11}^2 u_1}{w^5} - \frac{u_{11} u_1}{w^2}  F^{ij}a_{ij} = (f)_1.
\end{equation}
Note that
$(f)_1 = f_{x_1} + f_u u_1 + f_{\nu_j}(\nu_j)_1$.
By \eqref{D1Q} and $f_u \geqslant 0$, we have
\begin{equation}
\label{df}
\frac{(f)_1}{u_1} \geqslant -\frac{C}{u_1} + \frac{u_{11}f_{\nu_1}}{u_1 w} \Big(\frac{u_1^2}{w^2} - 1\Big) - f_{\nu_{n+1}} \frac{ u_{11}}{w^3} \geqslant -\frac{CA}{u_1},
\end{equation}
where $C$ depends on $|Df|_{C^0}\equiv \sup_{\bar\Omega\times[-C_0, C_0]\times \mathbb{S}^n} |Df|$.
Substituting \eqref{da} and \eqref{df} into \eqref{D2Q}, we have
\begin{equation}
\label{gradient-1}
\begin{aligned}
0 \geqslant&\; \frac{2F^{11} u_{11}^2 }{w^5} + \frac{u_{11} }{w^2} F^{ij}a_{ij} - \frac{CA}{u_1} - \frac{\hat F^{11 } u_{11}^2}{u_1^2}
+ A \hat F^{ij} u_{ij}\\
\geqslant &\; (\frac{2u_1^2}{w^2} - 1) \frac{F^{11} u_{11}^2 }{w^3 u_1^2 }  + (1 - \frac{ u_1^2 }{w^2} ) A C_n^p f - \frac{CA}{u_1},
\end{aligned}
\end{equation}
where in the last inequality we used $F^{ij}a_{ij} = \hat F^{ij} u_{ij} = C_n^p f$.
From \eqref{D1Q}, we see $u_{11} = - A u_1^2 < 0$.
By Lemma \ref{Dinew1} (4) and (2), there exists a positive constant $c_2$ depending on $n$, $p$ and $\inf f$ such that
\begin{equation}
\label{gradient-3}
0 \geqslant  \frac{c_2 A^2 u_{1}^4 }{2w^5} - \frac{CA}{u_1},
\end{equation}
which is a contradiction to large $A$ and $|Du|$.
Therefore, $Q$ attains its maximum on $\partial \Omega$ and hence \eqref{gradient} is proved.

\end{proof}

\begin{remark}
The interior gradient estimate was proved by Li \cite{L} for Weingarten curvature equations of general type by
the normal perturbation method.
\end{remark}

\end{document}